\documentclass[11pt, reqno]{amsart}

\usepackage[utf8]{inputenc}
\usepackage{indentfirst}
\usepackage{amsmath, amssymb, amsthm}
\usepackage{mathrsfs}
\usepackage{setspace}
\usepackage{enumerate}
\usepackage{wasysym}
\usepackage{graphicx}
\usepackage{tikz}
\usepackage{float}
\usepackage{booktabs}
\usepackage{pgfplots}
\usepackage{xcolor}
\usepackage[colorlinks=true,linkcolor=purple,citecolor=blue,urlcolor=magenta]{hyperref}

\definecolor{sectionlink}{RGB}{0,100,200}

\newtheorem{theo}{Theorem}[section]
\newtheorem{lem}{Lemma}[section]

\newtheorem{rem}{Remark}[section]

\newcommand{\ol}{\overline}

\numberwithin{equation}{section}
\begin{document}
\title[Sharp coefficient Estimates for the class $\mathcal{S}_{\mathcal{AP}}^{*}$]{Sharp coefficient Estimates for the class $\mathcal{S}_{\mathcal{AP}}^{*}$}
\author[P. Das, N. Sarkar]{Pradip Das, Nabadwip Sarkar}
\address{Department of Mathematics, Raiganj University, Raiganj, West Bengal-733134, India.}
\email{pradipsmath@gmail.com}
\address{Amity School of Applied Sciences, Amity University Mumbai, Panvel, Navi Mumbai, Maharashtra-410206, India.}
\email{naba.iitbmath@gmail.com}

\makeatletter
\@namedef{subjclassname@2020}{\textup{2020} Mathematics Subject Classification}
\makeatother


\renewcommand{\thefootnote}{\arabic{footnote}}
\setcounter{footnote}{0}
\begin{abstract}
We investigate several classic coefficient problems for the Ma--Minda starlike subclass $\mathcal{S}_{\mathcal{AP}}^{*}$ defined by the apple-like subordination function $\psi_{\mathcal{AP}}(z)=e^{z}\sqrt{1+z}$. Sharp bounds are derived for the initial inverse logarithmic coefficients $\Gamma_1$, $\Gamma_2$, $\Gamma_3$, and the successive modulus difference $|\Gamma_2|-|\Gamma_1|$. In addition, we evaluate the second-order inverse logarithmic Hankel determinant, the generalized Fekete--Szeg\"o functional over all real parameter domains, and the third-order Hermitian--Toeplitz determinant. The corresponding extremal functions are explicitly determined for each functional.
\end{abstract}
\maketitle
\section{{\bf Introduction}}
Let $\mathcal{H}$ denote the class of analytic functions defined in the open unit disk $\mathbb{D}:=\{z\in\mathbb{C}:|z|<1\}$, equipped with the topology of uniform convergence on compact subsets. We denote by $\mathcal{A}$ the subclass of $\mathcal{H}$ consisting of functions normalized by $f(0)=0$ and $f'(0)=1$. Let $\mathcal{S}$ be the family of functions $f\in\mathcal{A}$ that are univalent in $\mathbb{D}$. Every function $f\in\mathcal{S}$ can be expressed by the Taylor series expansion
\begin{equation}\label{eq1}
f(z)=z+\sum_{n=2}^{\infty}a_nz^n, \quad z\in\mathbb{D}.
\end{equation}

A function $f\in\mathcal{A}$ is starlike if the image domain $f(\mathbb{D})$ is starlike with respect to the origin, and is convex if $f(\mathbb{D})$ is a convex domain. These standard subclasses of $\mathcal{S}$ are denoted by $\mathcal{S}^{\ast}$ and $\mathcal{C}$, respectively. Analytically, a function $f\in\mathcal{A}$ belongs to $\mathcal{S}^{\ast}$ if and only if $\Re\left(zf'(z)/f(z)\right)>0$ for all $z\in\mathbb{D}$, and belongs to $\mathcal{C}$ if and only if $\Re\left(1+zf''(z)/f'(z)\right)>0$ for all $z\in\mathbb{D}$. The classical Alexander relation states that $f\in\mathcal{C}$ if and only if $zf'\in\mathcal{S}^{\ast}$.

Recall that an analytic function $f$ is subordinate to an analytic function $g$ in $\mathbb{D}$, written as $f\prec g$, if there exists a Schwarz function $\omega$ (analytic in $\mathbb{D}$ with $\omega(0)=0$ and $|\omega(z)|<1$) such that $f(z)=g(\omega(z))$ for $z\in\mathbb{D}$. In particular, if $g$ is univalent in $\mathbb{D}$, then $f\prec g$ is equivalent to $f(0)=g(0)$ and $f(\mathbb{D})\subset g(\mathbb{D})$.

Using the principle of subordination, Ma and Minda \cite{MM} introduced a unified framework for subclasses of starlike functions, defined by
\[ 
\mathcal{S}^{\ast}(\psi) = \left\{ f\in\mathcal{S}: \frac{zf'(z)}{f(z)} \prec \psi(z), \quad z\in\mathbb{D}\right\},
\] 
where $\psi$ is an analytic and univalent function mapping $\mathbb{D}$ onto a domain that is symmetric with respect to the real axis, with $\Re(\psi(z))>0$, $\psi(0)=1$, and $\psi'(0)>0$.

Various choices for the superordinate function $\psi$ have been studied extensively \cite{Arora2019, Arora2023, CK, STZ, rendi, Sm2026, periodica, TAA} to model particular geometric configurations and coefficient behavior. Recently, Sana, Saliu, and Riaz \cite{SanaSaliuRiaz} introduced the subclass $\mathcal{S}_{\mathcal{AP}}^{\ast}$ associated with the apple-like generating function $\psi_{\mathcal{AP}}(z)=e^{z}\sqrt{1+z}$. A function $f\in\mathcal{A}$ belongs to this class if 
\[\frac{zf'(z)}{f(z)}\prec\psi_{\mathcal{AP}}(z), \quad z\in\mathbb{D}.\]
The function $\psi_{\mathcal{AP}}$ maps the unit disk onto a symmetric, apple-like region and satisfies the standard Ma--Minda conditions, positioning $\mathcal{S}_{\mathcal{AP}}^{\ast}$ as a structural subclass of the broader Ma--Minda starlike family.

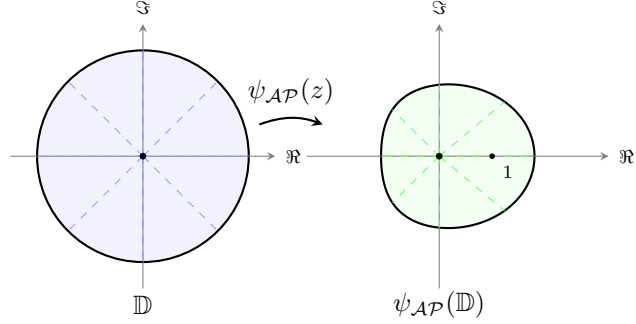
\begin{figure}[H]
\centering
\begin{tikzpicture}[scale=0.7, >=stealth]
\begin{scope}[shift={(-2.8,0)}]
    \draw[thick, fill=blue!5] (0,0) circle (2);
    \draw[->, gray, thin] (-2.5,0) -- (2.5,0) node[right, black] {\tiny $\Re$};
    \draw[->, gray, thin] (0,-2.5) -- (0,2.5) node[above, black] {\tiny $\Im$};
    \foreach \a in {0,45,...,315} {
        \draw[dashed, blue!40, thin] (0,0) -- (\a:2);
    }
    \filldraw[black] (0,0) circle (1.5pt);
    \node at (0,-2.8) {\small $\mathbb{D}$};
\end{scope}

\draw[->, thick, bend left=20] (-0.6,0.6) to node[above=1pt] {\small $\psi_{\mathcal{AP}}(z)$} (0.6,0.6);

\begin{scope}[shift={(2.8,0)}]
    \draw[thick, fill=green!5, smooth cycle, samples=200, variable=\t, domain=0:360]
    plot ({ (1.45 + 0.35*cos(\t) + 0.1*sin(\t)*sin(\t))*cos(\t) }, { (1.35 + 0.15*cos(\t))*sin(\t) });
    \draw[->, gray, thin] (-2.5,0) -- (3.2,0) node[right, black] {\tiny $\Re$};
    \draw[->, gray, thin] (0,-2.5) -- (0,2.5) node[above, black] {\tiny $\Im$};
    \foreach \a in {0,45,...,315} {
        \draw[dashed, green!60, thin] (0,0) -- ({ (1.45 + 0.35*cos(\a) + 0.1*sin(\a)*sin(\a))*cos(\a) }, { (1.35 + 0.15*cos(\a))*sin(\a) });
    }
    \filldraw[black] (1,0) circle (1.2pt) node[below right] {\tiny $1$};
    \filldraw[black] (0,0) circle (1.5pt);
    \node at (0,-2.8) {\small $\psi_{\mathcal{AP}}(\mathbb{D})$};
\end{scope}
\end{tikzpicture}
\caption{Conformal mapping of the open unit disk $\mathbb{D}$ onto the symmetric apple-like Ma--Minda target domain under the generating boundary function $\psi_{\mathcal{AP}}(z) = e^z\sqrt{1+z}$.}
\label{fig:apple_domain}
\end{figure}

The primary objective of this paper is to investigate several extremal problems within the class $\mathcal{S}_{\mathcal{AP}}^{\ast}$. Specifically, we study its inverse logarithmic coefficients, consecutive coefficient differences, and structural determinants to map the geometric bounds imposed by this target domain.

\subsection{Inverse Logarithmic Coefficients}
For a function $f\in\mathcal{S}$, the inverse logarithmic coefficients $\Gamma_n$ are defined by Ponnusamy et al.~\cite{Ponnusamy2018} via the log-expansion of the inverse function $f^{-1}$:
\[
F_{f^{-1}}(w):= \log\frac{f^{-1}(w)}{w} = 2\sum_{n=1}^{\infty}\Gamma_n w^n, \quad |w|<\frac{1}{4}.
\]
By explicit computation, the first four coefficients are related to the Taylor coefficients of $f$ by
\begin{equation}\label{IG1}
\begin{cases}
\Gamma_1 = -\frac{1}{2} a_2, \\
\Gamma_2 = -\frac{1}{2} a_3 + \frac{3}{4} a_2^2, \\
\Gamma_3 = -\frac{1}{2} \left( a_4 - 4a_2 a_3 + \frac{10}{3} a_2^3 \right), \\
\Gamma_4 = \frac{35}{8} a_2^4 - \frac{15}{2} a_2^2 a_3 + \frac{5}{2}a_2 a_4 + \frac{5}{4} a_3^2 - \frac{1}{2}a_5.
\end{cases}
\end{equation}
Ponnusamy et al.~\cite{Ponnusamy2018} proved that for the full class $\mathcal{S}$, the sharp inequality $|\Gamma_n| \le \frac{1}{2n}\binom{2n}{n}$ holds for all $n\in\mathbb{N}$, with equality if and only if $f$ is the Koebe function or one of its rotations. 

Determinants whose entries are logarithmic coefficients provide useful structural properties of univalent functions. Kowalczyk and Lecko \cite{12,15} studied the classical Hankel determinant $H_{q,n}(F_f/2)$ for logarithmic coefficients. Its inverse analogue, $H_{q,n}(F_{f^{-1}}/2)$, utilizes the inverse logarithmic coefficients $\Gamma_n$ as entries (see \cite{10,16}):
\[
H_{q,n}\left(F_{f^{-1}}/2\right)
=
\begin{vmatrix}
\Gamma_n & \Gamma_{n+1} & \cdots & \Gamma_{n+q-1} \\
\Gamma_{n+1} & \Gamma_{n+2} & \cdots & \Gamma_{n+q} \\
\vdots & \vdots & \ddots & \vdots \\
\Gamma_{n+q-1} & \Gamma_{n+q} & \cdots & \Gamma_{n+2(q-1)}
\end{vmatrix}.
\]

While the standard Taylor and logarithmic coefficients for functions associated with the apple-like domain have recently been examined by Sana et al.~\cite{SanaSaliuRiaz}, several corresponding problems for the inverse configurations remain open. To bridge this gap, this paper establishes the following sharp results for the class $\mathcal{S}_{\mathcal{AP}}^{\ast}$:
\begin{itemize}
\item Upper bounds for the initial inverse logarithmic coefficients $\Gamma_1$, $\Gamma_2$, and $\Gamma_3$.
\item Estimates for the second-order inverse logarithmic Hankel determinant $H_{2,1}(F_{f^{-1}}/2) = \Gamma_1\Gamma_3-\Gamma_2^2$.
\item Upper and lower bounds for the consecutive coefficient difference $|\Gamma_2|-|\Gamma_1|$.
\item Sharp bounds for the generalized Fekete--Szeg\"o functional $|a_3 - \lambda a_2^2| - \mu |a_2|$ over all real parameters $\lambda$ and $\mu > 0$.
\item Sharp bounds for the third-order Hermitian--Toeplitz determinant $T_{3,1}(f)$.
\end{itemize}
\section{{\bf Auxiliary lemmas}}
Let $\mathcal{P}$ denote the class of all analytic functions $p$ in the unit disk $\mathbb{D}$ satisfying 
$p(0) = 1$ and $\Re p(z) > 0$ for all $z \in \mathbb{D}$. 
Then, every $p \in \mathcal{P}$ admits the series representation
\begin{equation}\label{p1}
p(z) = 1 + \sum_{n=1}^{\infty} c_n z^n, \quad z \in \mathbb{D}.
\end{equation}
Functions in $\mathcal{P}$ are referred to as \emph{Carath\'{e}odory functions}. It is well-known that for $p \in \mathcal{P}$, the coefficients satisfy the sharp bound $|c_n| \leq 2$ for all $n \geq 1$ (see \cite{PLD1}). 
The Carath\'{e}odory class $\mathcal{P}$ and its coefficient bounds play a fundamental role 
in deriving estimates for sharp bounds in geometric function theory.\par 
Now, we state some lemmas, which will be useful to establish our main results:
Now we recall the following well-known result due to Cho et al. \cite{C12}.
\begin{lem}\label{L1} \cite[Lemma 2.4]{C12} If $p\in\mathcal{P}$ is of the form (\ref{p1}), then
\begin{equation}\label{c1}c_1 =2\tau_1,\end{equation}
\begin{equation}\label{c2} c_2=2\tau_1^2 + 2(1 - |\tau_1|^2)\tau_2\end{equation}
and
\begin{equation}\label{c3} c_3 = 2\tau_1^3+4(1-|\tau_1|^2)\tau_1\tau_2 - 2(1 - |\tau_1|^2)\ol{\tau_1}\tau_2^2 + 2(1 - \tau_1^2)(1 - |\tau_2|^2)\tau_3
\end{equation}
for some $\tau_1, \tau_2, \tau_3 \in\mathbb{\ol D}:= \{z \in \mathbb{C}: |z| \leq 1 \}$.
For $ \tau_1 \in \mathbb{T} := \{ z \in \mathbb{C} : |z| = 1 \} $ , there is a unique function $ p \in \mathcal{P} $ with $ c_1 $ as in (\ref{c1}), namely,
\[
p(z) = \frac{1 + \tau_1 z}{1 - \tau_1 z}, \quad z \in \mathbb{D}.
\]
For $ \tau_1 \in \mathbb{D} $ and $ \tau_2 \in \mathbb{T} $ , there is a unique function $ p \in \mathcal{P} $ with $ c_1 $ and $ c_2 $ as in (\ref{c1}) and (\ref{c2}), namely,
\[
p(z) = \frac{1 + (\ol \tau_1 \tau_2 + \tau_1) z + \tau_2 z^2}{1 + (\ol \tau_1 \tau_2 - \tau_1) z - \tau_2 z^2}, \quad z \in \mathbb{D}.
\]
For $ \tau_1, \tau_2 \in \mathbb{D} $ and $ \tau_3 \in \mathbb{T} $ , there is a unique function $ p \in \mathcal{P} $ with $ c_1 $ , $ c_2 $ , and $ c_3 $ as in (\ref{c1}-\ref{c3}), namely,
\[
p(z) = \frac{1 + (\ol\tau_2 \tau_3 + \ol\tau_1 \tau_2 + \tau_1)z+(\ol\tau_1\tau_3+\tau_1\ol\tau_2\tau_3+\tau_2)z^2+\tau_3z^3}{1+(\ol\tau_2\tau_3+\ol\tau_1\tau_2-\tau_1)z+(\ol\tau_1\tau_3-\tau_1\ol\tau_2\tau_3-\tau_2)z^2-\tau_3z^3},\;\;z\in\mathbb{D}
\]
\end{lem}

\medskip
Following well-known result is due to Choi et al. \cite{CKS1}.
\begin{lem}\label{L2}\cite{CKS1} Let $A$, $B$, $C$ be real numbers and let
\[Y(A, B, C):= \max\limits_{z\in \ol{\mathbb{D}}}\left\lbrace |A+Bz+Cz^2|+1-|z|^2\right\rbrace.\]
\begin{enumerate} 
\item[(i)] If $AC\geq 0$, then
\[Y(A, B, C) =
\begin{cases}
|A|+|B|+|C|, & \text{if}\;\;\; |B|\geq 2(1-|C|), \\
1+|A|+\frac{B^2}{4(1-|C|)}, &\text{if}\;\;\; |B|<2(1-|C|).
\end{cases}
\]
\item[(ii)] If $AC<0$, then 
\[Y(A,B,C)=
\begin{cases}
1-|A|+\frac{B^2}{4(1-|C|)}, &\text{if}\;\;\; -4AC(C^{-2}-1) \leq B^2\; \text{and}\; |B|<2(1-|C|), \\
1+|A|+\frac{B^2}{4(1+|C|)}, &\text{if}\;\;\; B^2<\min\left\{4(1+|C|)^2, -4AC(C^{-2}-1) \right\}, \\
R(A,B,C), &\text{otherwise},
\end{cases}
\]
where
\[R(A,B,C):=
\begin{cases}
|A|+|B|-|C|, & \text{if}\;\;\; |C|(|B|+4|A|) \leq |AB|, \\
-|A|+|B|+|C|, & \text{if}\;\;\; |AB|\leq |C|(|B|-4|A|), \\
(|C|+|A| )\sqrt{1-\frac{B^2}{4AC}}, &\text{otherwise}.
\end{cases}
\]
\end{enumerate} 
\end{lem}

\begin{lem}\label{L3} \cite{MM}
Let $p \in \mathcal{P}$ be given by \eqref{p1}. Then
\[
\left| c_2 - v c_1^2 \right| \le 
\begin{cases}
-4v + 2, & v < 0, \\
2, & 0 \leq v \leq 1, \\
4v - 2, & v > 1.
\end{cases}
\]
Moreover, for $v < 0$ or $v > 1$, equality holds if and only if
\[
h(z) = \frac{1+z}{1-z} \quad \text{or one of its rotations}.
\]
For $0 < v < 1$, equality holds if and only if
\[
h(z) = \frac{1+z^2}{1-z^2} \quad \text{or one of its rotations}.
\]
\end{lem}

\begin{lem}\label{L6}\cite{SimThomas2020}
Let $J, K,$ and $L$ be numbers such that $J \geq 0$, $K \in \mathbb{C}$, and $L \in \mathbb{R}$. 
Let $p \in \mathcal{P}$ be of the form (\ref{p1}) and define a function by
\[
\Phi(c_1,c_2) = \big| K c_1^2 + L c_2 \big| - \big| J c_1 \big|.
\]
Then 
\[
\Phi(c_{1}, c_{2}) \le 
\begin{cases}
|4K + 2L| - 2J, & \text{if } |2K + L| \geq |L| + J, \\[6pt]
2|L|, & \text{otherwise.}
\end{cases}
\]
and
\[
-\Phi(c_1,c_2) \leq 
\begin{cases}
2J - M, & \text{when } J \geq M + 2|L|, \\[6pt]
2J \sqrt{\dfrac{ 2|L|}{M + 2|L|}}, & \text{when } J^2 \leq 2|L|(M + 2|L|), \\[10pt]
2|L| +\dfrac{ J^2}{M + 2|L|}, & \text{otherwise}
\end{cases}
\]
where $M=|4K+2L|$.
\end{lem}
\section{{\bf Sharp bounds for the logarithmic inverse coefficients of the class $\mathcal{S}_{\mathcal{AP}}^{*}$.}} 
\begin{theo}\label{T1}
Let $f \in \mathcal{S}_{\mathcal{AP}}^{*}$, and let the inverse logarithmic coefficients $\Gamma_n$ ($n \geq 1$) be defined by \eqref{IG1}. Then
\[
|\Gamma_1| \le \frac{3}{4},
\qquad
|\Gamma_2| \le \frac{29}{32},
\qquad
|\Gamma_3| \le \frac{925}{576}.
\]
All these estimates are sharp.
\end{theo}

\begin{proof} 
Since $f \in \mathcal{S}_{\mathcal{AP}}^{*}$, there exists a Schwarz function $w(z) = \frac{p(z)-1}{p(z)+1}$ with $p \in \mathcal{P}$ such that $\frac{zf'(z)}{f(z)} = e^{w(z)}\sqrt{1+w(z)}$. Expanding $w(z)$ in powers of $z$ using the coefficients $c_n$ of $p$ yields
\[ w(z) = \frac{c_1}{2}z + \left( \frac{c_2}{2} - \frac{c_1^2}{4} \right)z^2 + \left( \frac{c_3}{2} - \frac{c_1c_2}{2} + \frac{c_1^3}{8} \right)z^3 + \cdots \]
Substituting this into the power series expansion of the subordination function gives
\[ e^{w(z)}\sqrt{1+w(z)} = 1 + \frac{3}{2}w(z) + \frac{7}{8}w^2(z) + \frac{17}{48}w^3(z) + \frac{19}{128}w^4(z) + \cdots \]
Equating coefficients in $zf'(z) = f(z)\left(1 + \sum_{n=1}^{\infty} q_n z^n\right)$, we obtain the Taylor coefficients of $f$:
\begin{equation}\label{pn1}
\begin{aligned}
    a_2 &= \frac{3}{4}c_1, \\
    a_3 &= \frac{3}{8}c_2 + \frac{13}{64}c_1^2, \\
    a_4 &= \frac{1}{4}c_3 + \frac{17}{96}c_1c_2 + \frac{37}{2304}c_1^3.
\end{aligned}
\end{equation}

By the definition $\Gamma_1 = -\frac{1}{2}a_2$ and \eqref{pn1}, we have
\[
|\Gamma_1| = \left|-\frac{1}{2}\cdot \frac{3}{4}c_1\right| = \frac{3}{8}|c_1| \le \frac{3}{4},
\]
since $|c_1| \le 2$ for all $p \in \mathcal{P}$. Equality holds if and only if $|c_1|=2$, which corresponds to the function
\begin{equation}\label{pn2}
f_0(z) = z\exp\left( \int_0^z \frac{e^t\sqrt{1+t}-1}{t}\,dt \right) = z + \frac{3}{2}z^2 + \frac{25}{16}z^3 + \frac{385}{288}z^4 + \cdots.
\end{equation}

From \eqref{IG1} and \eqref{pn1}, the second coefficient satisfies
\[
|\Gamma_2| = \left| -\frac{1}{2} a_3 + \frac{3}{4} a_2^2 \right| = \left| -\frac{1}{2}\left( \frac{3}{8} c_2 + \frac{13}{64}c_1^2 \right) + \frac{3}{4}\left( \frac{3}{4} c_1 \right)^2 \right| = \frac{3}{16} \left| c_2 - \frac{41}{24}c_1^2 \right|.
\]
Applying Lemma \ref{L3} with $v = \frac{41}{24} > 1$, we get $\left| c_2 - \frac{41}{24}c_1^2 \right| \le 4\left(\frac{41}{24}\right)-2 = \frac{29}{6}$. Therefore, $|\Gamma_2| \le \frac{3}{16} \cdot \frac{29}{6} = \frac{29}{32}$, which is attained by $f_0$.

To estimate $\Gamma_3$, we substitute the relations from \eqref{pn1} into the definition of $\Gamma_3$, which yields
\begin{equation}\label{G3}
|\Gamma_3| = \left| -\frac{1}{2} \left( a_4 - 4 a_2 a_3 + \frac{10}{3} a_2^3 \right) \right| = \frac{1}{8} \left| c_3 - \frac{91}{24} c_1 c_2 + \frac{1873}{576} c_1^3 \right|.
\end{equation}
Expressing the coefficients $c_n$ in terms of the parameters $\tau_1, \tau_2, \tau_3 \in \overline{\mathbb{D}}$ using Lemma \ref{L1} converts \eqref{G3} into
\begin{equation}\label{G11}
|\Gamma_3| = \frac{1}{8} \left| \frac{925}{72}\tau_1^3 - \frac{67}{6}\tau_1(1 - |\tau_1|^2)\tau_2 - 2\overline{\tau_1}(1 - |\tau_1|^2)\tau_2^2 + 2(1 - |\tau_1|^2)(1 - |\tau_2|^2)\tau_3 \right|.
\end{equation}
Because the class $\mathcal{P}$ and the functional are rotationally invariant, we assume without loss of generality that $c_1 \in [0,2]$, so that $\tau_1 \in [0,1]$. We consider two cases for $\tau_1$:

\noindent \textbf{Case 1.} If $\tau_1 = 1$, the terms containing $(1-|\tau_1|^2)$ vanish identically in \eqref{G11}, which immediately gives
\[
|\Gamma_3| = \frac{925}{576}.
\]

\noindent \textbf{Case 2.} If $0 \le \tau_1 < 1$, let $x = \tau_1 \in [0,1)$. Applying the triangle inequality to \eqref{G11} yields
\begin{equation}\label{G12}
|\Gamma_3| \le \frac{1 - x^2}{4} \left( \left| A + B \tau_2 + C \tau_2^2 \right| + 1 - |\tau_2|^2 \right),
\end{equation}
where
\begin{equation}\label{pn3}
A = \frac{925x^3}{144(1 - x^2)} > 0, \qquad B = -\frac{67}{12}x, \qquad C = -x.
\end{equation}
Since $AC < 0$, we apply part (ii) of Lemma \ref{L2}:

\noindent \textbf{Case 2(a).} If $-4AC(C^{-2}-1) \le B^2$ and $|B| < 2(1-|C|)$, the second condition simplifies to $\frac{67}{12}x < 2(1-x)$, which gives $x \in [0, \frac{24}{91})$. Since $-4AC(C^{-2}-1) = \frac{925}{36}x^2$ and $B^2 = \frac{4489}{144}x^2$, the inequality $\frac{925}{36} < \frac{4489}{144}$ holds on this interval. Lemma \ref{L2} gives the maximum value as $1 - |A| + \frac{B^2}{4(1-|C|)}$. Substituting this into \eqref{G12} yields
\[
|\Gamma_3| \le \frac{1-x^2}{4} \left( 1 - \frac{925x^3}{144(1-x^2)} + \frac{4489x^2}{576(1-x)} \right).
\]
Simplifying this expression gives the polynomial
\[
G(x) = \frac{789x^3 + 3913x^2 + 576}{2304}.
\]
Since $G'(x) = \frac{2367x^2 + 7826x}{2304} > 0$ for all $x > 0$, the function $G$ is strictly increasing on $[0, \frac{24}{91})$, and its maximum value is
\[
G\left( \frac{24}{91} \right) \approx 0.374 < \frac{925}{576}.
\]

\noindent \textbf{Case 2(b).} If $B^2 < \min\left\{4(1+|C|)^2, -4AC(C^{-2}-1)\right\}$, the condition requires $B^2 < -4AC(C^{-2}-1)$. Substituting the parameters gives $\frac{4489}{144}x^2 < \frac{3700}{144}x^2$, which implies $4489 < 3700$. This contradiction shows that this case cannot occur for $x \in (0,1)$.

\noindent \textbf{Case 2(c).} We examine the condition $|C|(|B|+4|A|) \le |AB|$ from the remaining sub-cases of Lemma \ref{L2}. Substituting the absolute parameter values for $x \in (0,1)$ yields:
\[
x \left( \frac{67}{12}x + 4 \cdot \frac{925x^3}{144(1 - x^2)} \right) \le \frac{61975x^4}{1728(1 - x^2)},
\]
which simplifies directly to the polynomial form:
\[
9648x^2 - 27223x^4 \le 0.
\]
Dividing through by $x^2$ (since $x > 0$), the parameter constraint isolates to the sub-interval:
\[
27223x^2 \ge 9648 \implies x \ge \sqrt{\frac{9648}{27223}} \approx 0.5954.
\]
When this condition holds, Lemma \ref{L2} states that the maximum value is given by $R(A,B,C) = |A| + |B| - |C|$. Substituting this back into the main bounding relation \eqref{G12} fields:
\[
|\Gamma_3| \le \frac{1 - x^2}{4} \left( \frac{925x^3}{144(1 - x^2)} + \frac{67}{12}x - x \right).
\]
Simplifying we get:
\[
|\Gamma_3| \le \frac{925x^3}{576} + \frac{55x(1-x^2)}{48} = \frac{265x^3 + 660x}{576} =: \Phi(x).
\]
Differentiating $\Phi(x)$ with respect to $x$ yields:
\[
\Phi'(x) = \frac{795x^2 + 660}{576}.
\]
Since $\Phi'(x) > 0$ holds uniformly for all $x \in (0,1)$, the function $\Phi$ is strictly increasing on this sub-interval and reaches its upper bound at the right-hand endpoint:
\[
\lim_{x \to 1^-} \Phi(x) = \frac{265(1)^3 + 660(1)}{576} = \frac{925}{576}.
\]

\noindent \textbf{Case 2(d).} We evaluate the sub-case condition $|AB| \le |C|(|B|-4|A|)$ from part (ii) of Lemma \ref{L2}. Substituting the absolute parameters yields:
\[
\frac{61975x^4}{1728(1 - x^2)} \le x \left( \frac{67}{12}x - 4 \cdot \frac{925x^3}{144(1 - x^2)} \right).
\]
which simplifies to the polynomial inequality:
\[
116023x^4 - 9648x^2 \le 0.
\]
Hence, the validity criteria isolates to the sub-interval:
\[
116023x^2 \le 9648 \implies x \le \sqrt{\frac{9648}{116023}} \approx 0.2884.
\]
For parameters within this range, Lemma \ref{L2} dictates that the maximum value is given by $R(A,B,C) = -|A| + |B| + |C|$. Substituting this back into the main bounding relation \eqref{G12} we get:
\[
|\Gamma_3| \le \frac{1 - x^2}{4} \left( -\frac{925x^3}{144(1 - x^2)} + \frac{67}{12}x + x \right).
\]
Distributing the leading term $\frac{1-x^2}{4}$ across the brackets eliminates the rational denominator, leading to:
\[
|\Gamma_3| \le -\frac{925x^3}{576} + \frac{79x(1-x^2)}{48} = \frac{948x - 1873x^3}{576} =: \Phi_{\star }(x).
\]
Differentiating $\Phi_{\star }(x)$ with respect to $x$ yields:
\[
\Phi_{\star}'(x) = \frac{948 - 5619x^2}{576}.
\]
Setting $\Phi_{\star}'(x) = 0$ gives a critical point at $x = \sqrt{\frac{948}{5619}} \approx 0.4107$, which lies completely outside the active sub-interval $(0, 0.2884]$. Thus, $\Phi_{\star }(x)$ is strictly increasing on its domain and attains its maximum value at the boundary endpoint $x = 0.2884$:
\[
\Phi_{\star}(0.2884) = \frac{948(0.2884) - 1873(0.2884)^3}{576} \approx 0.3967.
\]
Since $0.3967 < \frac{925}{576} \approx 1.6059$, the global upper bound remains determined by Case 1.

\noindent \textbf{Case 2(e).} For the remaining sub-case within the "otherwise" branch of Lemma \ref{L2}, the maximum value is governed by the radical formulation $R(A,B,C) = (|C|+|A| ) \sqrt{1-\frac{B^2}{4AC}}$. Based on the domain constraints derived in Case 2(c) and Case 2(d), this branch is active on the interval:
\[
x \in \left( \sqrt{\frac{9648}{116023}}, \sqrt{\frac{9648}{27223}} \right) \approx (0.2884, 0.5954).
\]
We substitute the absolute parameters for $x \in (0,1)$ into the components of the radical formula. First, the multiplier term simplifies to:
\[
|C| + |A| = x + \frac{925x^3}{144(1 - x^2)} = \frac{x(144 + 781x^2)}{144(1 - x^2)}.
\]
Next, since $A > 0$ and $C < 0$, the expression inside the square root becomes:
\[
1 - \frac{B^2}{4AC} = 1 + \frac{|B|^2}{4|AC|} = 1 + \frac{\left(\frac{67}{12}x\right)^2}{4 \left(\frac{925x^3}{144(1 - x^2)}\right)x} = 1 + \frac{4489(1-x^2)}{3700x^2} = \frac{4489 - 789x^2}{3700x^2}.
\]
Taking the square root and combining the terms into $R(A,B,C)$ yields:
\[
R(A,B,C) = \left[ \frac{x(144 + 781x^2)}{144(1 - x^2)} \right] \cdot \left[ \frac{\sqrt{4489 - 789x^2}}{\sqrt{3700}x} \right] = \frac{(144 + 781x^2)\sqrt{4489 - 789x^2}}{144\sqrt{3700}(1-x^2)}.
\]
Substituting this maximum value back into the main bounding relation \eqref{G12} allows for a perfect cancellation of the $(1-x^2)$ terms:
\[
|\Gamma_3| \le \frac{1 - x^2}{4} \cdot \left[ \frac{(144 + 781x^2)\sqrt{4489 - 789x^2}}{144\sqrt{3700}(1-x^2)} \right] = \frac{144+781x^2}{576} \sqrt{\frac{4489-789x^2}{3700}} =: \Omega(x).
\]
To show that $\Omega(x) < \frac{925}{576}$ holds on this active domain, we substitute $t = x^2 \in (0,1)$ and examine the derivative of the squared numerator polynomial $P(t) = (144+781t)^2(4489-789t)$. Applying the product rule gives:
\[
P'(t) = (144+781t) \left[ 1562(4489-789t) - 789(144+781t) \right] = (144+781t)(6898202 - 1854939t).
\]
Since $(144+781t) > 0$ and $(6898202 - 1854939t) > 0$ for all $t \in (0,1)$, it follows that $P'(t) > 0$, meaning $\Omega(x)$ is strictly increasing. It bounded above by its right-hand limiting endpoint:
\[
\lim_{x \to 1^-} \Omega(x) = \frac{144+781(1)}{576} \sqrt{\frac{4489-789(1)}{3700}} = \frac{925}{576}.
\]
Thus, the maximum value across all sub-cases of Case 2 remains securely bounded by $\frac{925}{576}$, matching Case 1.
\end{proof}

\subsection{Sharpness and Geometric Extremality}
To verify the sharpness of the bounds obtained in Theorem \eqref{T1}, we examine the behavior of the rotationally invariant function $f_0 \in \mathcal{S}_{\mathcal{AP}}^{*}$ defined in \eqref{pn2}, which satisfies the structural differential relation
\begin{equation}\label{pn4}
\frac{zf_0'(z)}{f_0(z)} = e^z \sqrt{1+z}.
\end{equation}
The analytic mapping $\psi_{\mathcal{AP}}(z) = e^z \sqrt{1+z}$ maps the open unit disk $\mathbb{D}$ onto a bounded, slit-free apple-like domain symmetric with respect to the real axis. Integrating the power series expansion of \eqref{pn4} gives the initial Taylor coefficients of $f_0(z)$:
\[
a_2 = \frac{3}{2}, \qquad a_3 = \frac{25}{16}, \qquad a_4 = \frac{385}{288}, \qquad a_5 = \frac{383}{384}.
\]
Substituting these values directly into the definitions of the inverse logarithmic coefficients $\Gamma_n$ given in \eqref{IG1} yields:
\begin{itemize}
\item For $\Gamma_1$: 
\[ \Gamma_1 = -\frac{1}{2}a_2 = -\frac{3}{4} \implies |\Gamma_1| = \frac{3}{4}. \]
\item For $\Gamma_2$: 
\[ \Gamma_2 = -\frac{1}{2}a_3 + \frac{3}{4}a_2^2 = -\frac{25}{32} + \frac{27}{16} = \frac{29}{32} \implies |\Gamma_2| = \frac{29}{32}. \]
\item For $\Gamma_3$: 
\[ \Gamma_3 = -\frac{1}{2}\left(a_4 - 4a_2 a_3 + \frac{10}{3}a_2^3\right) = -\frac{1}{2}\left(\frac{385}{288} - \frac{75}{8} + \frac{45}{4}\right) = -\frac{925}{576} \implies |\Gamma_3| = \frac{925}{576}. \]
\end{itemize}
In each case, the absolute value of the coefficient matches the upper bound precisely, proving sharpness. The conformal deformation from the unit disk $\mathbb{D}$ to the univalent image domain $f_0(\mathbb{D})$ is shown in Figure~\ref{fig:extremal_geometry}.

\begin{figure}[H]
\centering
\begin{tikzpicture}[scale=0.7, >=stealth]
\begin{scope}[shift={(-2.8,0)}]
    \draw[thick, fill=blue!5] (0,0) circle (2);
    \draw[->, gray, thin] (-2.5,0) -- (2.5,0) node[right, black] {\tiny $\Re$};
    \draw[->, gray, thin] (0,-2.5) -- (0,2.5) node[above, black] {\tiny $\Im$};
    \foreach \a in {0,45,...,315} {
        \draw[dashed, blue!40, thin] (0,0) -- (\a:2);
    }
    \filldraw[black] (0,0) circle (1.5pt);
    \node at (0,-2.8) {\small $\mathbb{D}$};
\end{scope}

\draw[->, thick, bend left=20] (-0.6,0.6) to node[above=1pt] {\small $f_0$} (0.6,0.6);

\begin{scope}[shift={(2.8,0)}]
    \draw[thick, fill=orange!5, smooth cycle, samples=200, variable=\t, domain=0:360]
    plot ({ (1.55+0.35*cos(\t)+0.18*cos(3*\t))*cos(\t) }, { (1.45+0.28*cos(\t))*sin(\t) });
    \draw[->, gray, thin] (-2.5,0) -- (3.2,0) node[right, black] {\tiny $\Re$};
    \draw[->, gray, thin] (0,-2.5) -- (0,2.5) node[above, black] {\tiny $\Im$};
    \foreach \a in {0,45,...,315} {
        \draw[dashed, orange!60, thin] (0,0) -- ({ (1.55+0.35*cos(\a)+0.18*cos(3*\a))*cos(\a) }, { (1.45+0.28*cos(\a))*sin(\a) });
    }
    \filldraw[black] (0,0) circle (1.5pt);
    \node at (0,-2.8) {\small $f_0(\mathbb{D})$};
\end{scope}
\end{tikzpicture}
\caption{Conformal mapping profile from the open unit disk $\mathbb{D}$ to the starlike target domain $f_0(\mathbb{D})$ associated with the sharp boundary constants.}
\label{fig:extremal_geometry}
\end{figure}
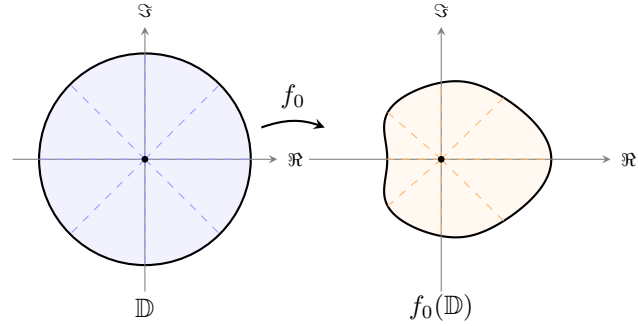

\section{\bf Bounds for the differences of logarithmic inverse coefficients for $\mathcal{S}_{\mathcal{AP}}^{*}$}
The Bieberbach conjecture, proved by de Branges~\cite{LDB1}, states that the Taylor coefficients of any function $f \in \mathcal{S}$ of the form~\eqref{eq1} satisfy the bound $|a_n| \leq n$ for all $n \geq 2$, with equality holding only for the Koebe function and its rotations. This result motivated extensive study into the behavior of differences of consecutive coefficients. For instance, the inequality 
\[
\bigl||a_{n+1}| - |a_n|\bigr| \leq 1, \quad n \geq 2,
\]
was shown to hold for starlike functions by Leung~\cite{Leung1978}, following a conjecture by Pommerenke~\cite{Pommerenke1971}. For convex functions, similar coefficient differences were investigated by Li and Sugawa~\cite{LiSugawa}.

Recently, attention has focused on the differences of logarithmic coefficients $|\gamma_{n+1}| - |\gamma_n|$ (see Lecko and Partyka~\cite{Lecko2023}, Obradovi\'{c} and Tuneski~\cite{Obradovic2024}, and Kumar and Cho~\cite{Kumar2023}). In this section, we determine sharp upper and lower bounds for the difference of the initial inverse logarithmic coefficients, $|\Gamma_2| - |\Gamma_1|$, for functions in the class $\mathcal{S}_{\mathcal{AP}}^{*}$.

\begin{theo}
Let $f \in \mathcal{S}_{\mathcal{AP}}^{*}$ and let the inverse logarithmic coefficients $\Gamma_n$ ($n = 1,2$) be defined by \eqref{IG1}. Then
\begin{equation}\label{T2}
-\frac{3\sqrt{3}}{2\sqrt{41}} \le |\Gamma_2|-|\Gamma_1| \le \frac{3}{8}.
\end{equation}
The inequalities are sharp.
\end{theo}

\begin{proof}
Let $f \in \mathcal{S}_{\mathcal{AP}}^{*}$. From the relation between the inverse logarithmic coefficients and the Carath\'eodory coefficients, we have
\begin{equation}\label{pn6}
|\Gamma_2|-|\Gamma_1| = \frac{3}{16}\left|c_2-\frac{41}{24}c_1^2\right| -\frac{3}{8}|c_1| = \frac{1}{128}\left(|41c_1^2-24c_2|-48|c_1|\right) = \frac{1}{128}\Phi(c_1,c_2),
\end{equation}
where $\Phi(c_1,c_2) = |Kc_1^2+Lc_2|-|Jc_1|$ with $K=41$, $L=-24$, and $J=48$.

Calculating the auxiliary parameters for Lemma \ref{L6} gives $M = |4K+2L| = 116$. Since $|2K+L| = 58 < 72 = |L|+J$, the upper bound condition yields $\Phi(c_1,c_2) \le 2|L| = 48$. For the lower bound, because $J^2 = 2304 < 2|L|(M+2|L|) = 7872$, Lemma \ref{L6} gives $-\Phi(c_1,c_2) \le 2J \sqrt{\frac{2|L|}{M+2|L|}} = \frac{192\sqrt{3}}{\sqrt{41}}$. Substituting these bounds into \eqref{pn6} gives the inequalities stated in Theorem \eqref{T2}.
\end{proof}

\subsection{Sharpness and Extremal Domains}
To establish the sharpness of the upper bound in Theorem \eqref{T2}, we consider the function $f_1 \in \mathcal{S}_{\mathcal{AP}}^{*}$ defined by
\begin{equation}\label{mn7}
f_1(z) = z\exp\left( \int_0^z \frac{e^{t^2}\sqrt{1+t^2}-1}{t}\,dt \right).
\end{equation}
This maps to the Schwarz function $\omega(z)=z^2$, which corresponds to $p(z) = \frac{1+z^2}{1-z^2} = 1+2z^2+2z^4+\cdots$. Thus, $c_1=0$ and $c_2=2$. This yields $a_2=0$ and $a_3=\frac{3}{4}$, which implies $\Gamma_1 = 0$ and $\Gamma_2 = -\frac{3}{8}$. Consequently, $|\Gamma_2|-|\Gamma_1| = \frac{3}{8}$. The geometric structure of the two-fold symmetric domain $f_1(\mathbb{D})$ is shown in Figure~\ref{fig:extremal_f1_geometry}.

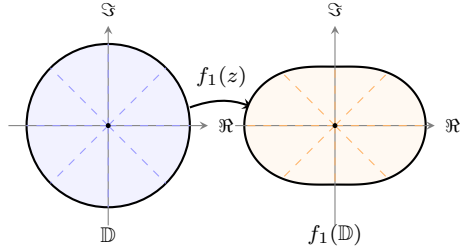
\begin{figure}[H]
\centering
\begin{tikzpicture}[scale=0.6, >=stealth]
\begin{scope}[shift={(-2.5,0)}]
    \draw[thick, fill=blue!5] (0,0) circle (1.8);
    \draw[->, gray, thin] (-2.2,0) -- (2.2,0) node[right, black] {\tiny $\Re$};
    \draw[->, gray, thin] (0,-2.2) -- (0,2.2) node[above, black] {\tiny $\Im$};
    \foreach \a in {0,45,...,315} { \draw[dashed, blue!40, thin] (0,0) -- (\a:1.8); }
    \filldraw[black] (0,0) circle (1.2pt); 
    \node at (0,-2.4) {\scriptsize $\mathbb{D}$};
\end{scope}

\draw[->, thick, bend left=20] (-0.7,0.4) to node[above=1pt] {\scriptsize $f_1(z)$} (0.7,0.4);

\begin{scope}[shift={(2.5,0)}]
    \draw[thick, fill=orange!5, smooth cycle, samples=120, variable=\t, domain=0:360] plot ({(1.65+0.35*cos(2*\t))*cos(\t)}, {(1.65+0.35*cos(2*\t))*sin(\t)});
    \draw[->, gray, thin] (-2.2,0) -- (2.2,0) node[right, black] {\tiny $\Re$};
    \draw[->, gray, thin] (0,-2.2) -- (0,2.2) node[above, black] {\tiny $\Im$};
    \foreach \a in {0,45,...,315} { \draw[dashed, orange!60, thin] (0,0) -- ({(1.65+0.35*cos(2*\a))*cos(\a)}, {(1.65+0.35*cos(2*\a))*sin(\a)}); }
    \filldraw[black] (0,0) circle (1.2pt); 
    \node at (0,-2.4) {\scriptsize $f_1(\mathbb{D})$};
\end{scope}
\end{tikzpicture}
\caption{Conformal mapping from the unit disk $\mathbb{D}$ onto the two-fold symmetric extremal domain $f_1(\mathbb{D})$.}
\label{fig:extremal_f1_geometry}
\end{figure}

To establish the sharpness of the lower bound in \eqref{T2}, we consider the function $f_2 \in \mathcal{S}_{\mathcal{AP}}^{*}$ defined by
\begin{equation}\label{mn8}
f_2(z) = z\exp\left( \int_0^z \frac{e^{\omega(t)}\sqrt{1+\omega(t)}-1}{t}\,dt \right),
\end{equation}
where $\omega(z) = z\left(z+\frac{2\sqrt{3}}{\sqrt{41}}\right) / \left(1+\frac{2\sqrt{3}}{\sqrt{41}}z\right)$. The coefficients of the corresponding function $p(z) = \frac{1+\omega(z)}{1-\omega(z)}$ are $c_1 = \frac{4\sqrt{3}}{\sqrt{41}}$ and $c_2=2$. This gives $\Gamma_1 = -\frac{3\sqrt{3}}{2\sqrt{41}}$ and $\Gamma_2 = 0$, which yields $|\Gamma_2|-|\Gamma_1| = -\frac{3\sqrt{3}}{2\sqrt{41}}$. The asymmetric layout of the domain $f_2(\mathbb{D})$ is shown in Figure~\ref{fig:extremal_f2_geometry}.

\begin{figure}[H]
\centering
\begin{tikzpicture}[scale=0.6, >=stealth]
\begin{scope}[shift={(-2.5,0)}]
    \draw[thick, fill=blue!5] (0,0) circle (1.8);
    \draw[->, gray, thin] (-2.2,0) -- (2.2,0) node[right, black] {\tiny $\Re$};
    \draw[->, gray, thin] (0,-2.2) -- (0,2.2) node[above, black] {\tiny $\Im$};
    \foreach \a in {0,45,...,315} { \draw[dashed, blue!40, thin] (0,0) -- (\a:1.8); }
    \filldraw[black] (0,0) circle (1.2pt); 
    \node at (0,-2.4) {\scriptsize $\mathbb{D}$};
\end{scope}

\draw[->, thick, bend left=20] (-0.7,0.4) to node[above=1pt] {\scriptsize $f_2(z)$} (0.7,0.4);

\begin{scope}[shift={(2.5,0)}]
    \draw[thick, fill=red!5, smooth cycle, samples=120, variable=\t, domain=0:360] plot ({(1.65+0.4*cos(\t)+0.12*cos(2*\t))*cos(\t)}, {(1.55+0.25*cos(\t))*sin(\t)});
    \draw[->, gray, thin] (-2.2,0) -- (2.6,0) node[right, black] {\tiny $\Re$};
    \draw[->, gray, thin] (0,-2.2) -- (0,2.2) node[above, black] {\tiny $\Im$};
    \foreach \a in {0,45,...,315} { \draw[dashed, red!60, thin] (0,0) -- ({(1.65+0.4*cos(\a)+0.12*cos(2*\a))*cos(\a)}, {(1.55+0.25*cos(\a))*sin(\a)}); }
    \filldraw[black] (0,0) circle (1.2pt); 
    \node at (0,-2.4) {\scriptsize $f_2(\mathbb{D})$};
\end{scope}
\end{tikzpicture}
\caption{Conformal mapping from the unit disk $\mathbb{D}$ onto the asymmetric extremal domain $f_2(\mathbb{D})$.}
\label{fig:extremal_f2_geometry}
\end{figure}
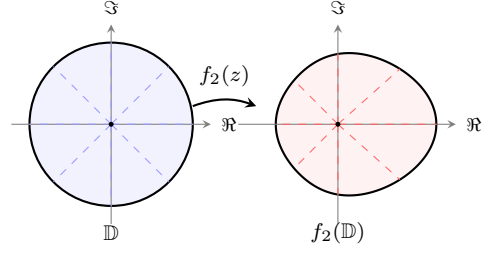

\section{{\bf Hankel determinants for the inverse logarithmic coefficients for $\mathcal{S}_{\mathcal{AP}}^{*}$}}
\begin{theo} 
Let $f\in \mathcal{S}_{\mathcal{AP}}^{*}$ be given by \eqref{eq1}. Then 
\[
\left| H_{2,1}\left(F_{f^{-1}} / 2\right)\right| \le \frac{1177}{3072}.
\] 
The inequality is sharp.
\end{theo}

\begin{proof} 
Substituting the inverse logarithmic transformations from \eqref{IG1} and the corresponding Taylor coefficients from \eqref{pn1} into the definition of the second-order Hankel determinant yields
\begin{equation}\label{mn9}
H_{2,1}(F_{f^{-1}} / 2) = \Gamma_1 \Gamma_3 - \Gamma_2^2 = \frac{2449}{49152} c_1^4 - \frac{59}{1024} c_1^2 c_2 + \frac{3}{64} c_1 c_3 - \frac{9}{256} c_2^2.
\end{equation}
By applying Lemma \ref{L1} to express the coefficients $c_n$ in terms of $\tau_n \in \overline{\mathbb{D}}$, \eqref{mn9} expands to
\begin{equation}\label{mn10}
|H_{2,1}(F_{f^{-1}} / 2)| = \left| \frac{1177}{3072}\tau_1^4 - \frac{1 - |\tau_1|^2}{128} \left( 47\tau_1^2\tau_2 + 24|\tau_1|^2\tau_2^2 - 24\tau_1(1-|\tau_2|^2)\tau_3 + 18(1-|\tau_1|^2)\tau_2^2 \right) \right|.
\end{equation}
Since the class $\mathcal{P}$ and the functional $H_{2,1}(F_{f^{-1}}/2)$ are invariant under rotational transformations, we assume without loss of generality that $c_1 \in [0,2]$, which implies $\tau_1 \in [0,1]$. We analyze the functional under three cases determined by the value of $\tau_1$:

\noindent\textbf{Case 1.} If $\tau_1=0$, then \eqref{mn10} simplifies to
\[
\left|H_{2,1}(F_{f^{-1}}/2)\right| \le \frac{18}{128}|\tau_2|^2 \le \frac{9}{64} = 0.140625.
\]

\noindent\textbf{Case 2.} If $\tau_1=1$, the term containing $(1-|\tau_1|^2)$ in \eqref{mn10} vanishes identically, leaving
\[
\left|H_{2,1}(F_{f^{-1}}/2)\right| = \frac{1177}{3072} \approx 0.383139.
\]

\noindent\textbf{Case 3.} If $\tau_1 \in (0,1)$, let $x = \tau_1 \in (0,1)$. Applying the triangle inequality to \eqref{mn10} yields
\begin{equation}\label{mn11}
\left| H_{2,1}(F_{f^{-1}} / 2) \right| \le \frac{3x (1 - x^2)}{16} \left( |A + B \tau_2 + C \tau_2^2| + 1 - |\tau_2|^2 \right),
\end{equation}
where
\[
A = \frac{1177x^3}{576(1-x^2)}, \qquad B=-\frac{47}{24}x, \qquad C=-\frac{3+x^2}{4x}.
\]
Since $A > 0$ and $C < 0$, we have $AC < 0$. We evaluate the conditions from part (ii) of Lemma \ref{L2}:

\noindent\textbf{Case 3(a).} The condition $|B|<2(1-|C|)$ cannot be satisfied because $|C| = \frac{3+x^2}{4x} > 1$ for all $x \in (0,1)$, which implies $2(1-|C|)<0$, while $|B| > 0$.

\noindent\textbf{Case 3(b).} The condition $B^2 < -4AC(C^{-2}-1)$ is not satisfied because $|C|>1$ implies $C^{-2}-1<0$, which makes the right-hand side negative while $B^2 \ge 0$.

\noindent\textbf{Case 3(c).} The condition $|C|(|B|+4|A|)\le |AB|$ simplifies to $33839x^4-71208x^2-20304\ge0$. Setting $t=x^2$, the unique positive root of this quadratic is $t_k \approx 2.3536$. Since $t_k > 1$, this inequality does not hold for $0<x<1$.

\noindent\textbf{Case 3(d).} We next evaluate the condition $|AB|\le |C|(|B|-4|A|)$ arising from the sub-cases of Lemma \ref{L2}. Substituting the explicit absolute parameter expressions for $x \in (0,1)$ yields:
\[
\left(\frac{1177x^3}{576(1-x^2)}\right) \left(\frac{47}{24}x\right) \le \left(\frac{3+x^2}{4x}\right) \left[ \frac{47}{24}x - 4\left(\frac{1177x^3}{576(1-x^2)}\right) , \right]
\]
which simplifies to the following polynomial inequality:
\[
90335x^4 + 98280x^2 - 20304 \le 0.
\]
By setting $t = x^2$, the unique positive root of the resulting quadratic equation $90335t^2 + 98280t - 20304 = 0$ is found to be $t_m \approx 0.1776$. This restricts the validity of this sub-case to the interval:
\[
0 < x \le x_m = \sqrt{t_m} \approx 0.4214.
\]
For parameters within this region, Lemma \ref{L2} states that the maximum value of the internal expression is given by $Y(A,B,C) = -|A|+|B|+|C|$. Substituting this back into the main bounding relation \eqref{mn11} yields:
\[
|H_{2,1}(F_{f^{-1}} / 2)| \le \frac{3x(1-x^2)}{16} \left( -\frac{1177x^3}{576(1-x^2)} + \frac{47}{24}x + \frac{3+x^2}{4x} \right).
\]
i.e.
\[
|H_{2,1}(F_{f^{-1}} / 2)| \le \frac{-2449x^4 + 840x^2 + 432}{3072} =: \Psi(x).
\]
An optimization of $\Psi(x)$ over the active sub-interval $(0, 0.4214]$ shows that the function increases strictly from its initial value $\Psi(0) = \frac{432}{3072} = 0.140625$ up to its boundary endpoint value:
\[
\Psi(0.4214) \approx 0.16408.
\]
Since $0.16408 < \frac{1177}{3072} \approx 0.38314$, the maximum value obtained in this interval is strictly bounded above, ensuring that the global supremum remains determined by the boundary case $\tau_1 = 1$.

\noindent\textbf{Case 3(e).} In the remaining sub-case where $Y(A,B,C) = (|A|+|C|) \sqrt{1-\frac{B^2}{4AC}}$, substituting the parameters into \eqref{mn11} gives $\left|H_{2,1}(F_{f^{-1}}/2)\right| \le \Lambda(x)$, where
\[
\Lambda(x) = \frac{1033x^4 -288x^2 +432}{3072} \sqrt{ \frac{5740-1032x^2}{1177(3+x^2)} }.
\]
Differentiating $\Lambda(x)$ shows that $\Lambda'(x)>0$ holds uniformly on $(0,1)$. Thus, $\Lambda$ is strictly increasing, and its supremum converges to the right-hand boundary limit:
\[
\lim_{x \to 1^-} \Lambda(x) = \frac{1177}{3072}.
\]
Therefore, $\left|H_{2,1}(F_{f^{-1}}/2)\right| \le \frac{1177}{3072}$, which matches Case 2.
\end{proof}

\subsection{Sharpness}
The upper bound $\frac{1177}{3072}$ is achieved by the rotationally invariant function $f_0 \in \mathcal{S}_{\mathcal{AP}}^{*}$ defined in \eqref{pn2}. For this function, the initial inverse logarithmic coefficients are
\[
\Gamma_1 = -\frac{3}{4}, \qquad \Gamma_2 = \frac{29}{32}, \qquad \Gamma_3 = \frac{925}{576}.
\]
Substituting these exact values into the determinant formula gives
\[
H_{2,1}\left(F_{f^{-1}} / 2\right) = \Gamma_1\Gamma_3 - \Gamma_2^2 = \left(-\frac{3}{4}\right)\left(\frac{925}{576}\right) - \left(\frac{29}{32}\right)^2 = -\frac{925}{768} - \frac{841}{1024} = -\frac{1177}{3072}.
\]
Taking the absolute value yields $\left| H_{2,1}\left(F_{f^{-1}} / 2\right) \right| = \frac{1177}{3072}$, confirming sharpness. 

\section{\bf Generalized Fekete-Szeg\H{o} Functional for $\mathcal{S}_{\mathcal{AP}}^{*}$}

The classical Fekete--Szeg\H{o} problem, which originates from the study of the univalent coefficient functional $|a_3 - \lambda a_2^2|$, serves as a delicate indicator of the local geometry of analytic maps under variation. In 2024, Lecko and Partyka \cite{Lecko2024} introduced a powerful generalization of this functional by incorporating a weighted first-order perturbation term, defining it as
\begin{equation}\label{FG}
F_{\lambda,\mu}(f) := \big|a_3-\lambda a_2^2\big| - \mu |a_2|, \qquad \lambda\in\mathbb{C}, \quad \mu>0,
\end{equation}
where $a_2$ and $a_3$ arise from the standard Taylor expansion given in \eqref{eq1}. This formulation was subsequently expanded to the full class of normalized univalent functions $\mathcal{S}$ and its convex subclass $\mathcal{K}$ by Bulboac\u{a} \textit{et al.} \cite{Bulboaca2025}.

The generalized functional $F_{\lambda,\mu}(f)$ measures the specific interaction between the second and third Taylor coefficients under a weighted linear penalty. Because the apple-like Ma--Minda geometry introduces non-classical coefficient relations that deviate strongly from standard half-plane or circular maps, it is natural to investigate the extremal behavior of this functional within the subclass $\mathcal{S}_{\mathcal{AP}}^{*}$. To maximize clarity for the reader and simplify the parametric regions, we present the sharp upper and lower bounds as two independent structural theorems.
\begin{theo}[Sharp Upper Estimates]\label{theo_FS_upper}
Let $f \in \mathcal{S}_{\mathcal{AP}}^{*}$ be of the form \eqref{eq1}. Then the generalized Fekete--Szeg\H{o} functional satisfies the upper bound $F_{\lambda,\mu}(f) \le U_{\mathcal{AP}}(\lambda,\mu)$, where
\[
U_{\mathcal{AP}}(\lambda,\mu) = 
\begin{cases}
\displaystyle \frac{9}{4} \left| \lambda - \frac{25}{36} \right| - \frac{3\mu}{2}, & \text{if } \displaystyle \left| \lambda - \frac{25}{36} \right| \ge \frac{1}{3} + \frac{2\mu}{3}, \\[4mm]
\displaystyle \frac{3}{4}, & \text{if } \displaystyle \left| \lambda - \frac{25}{36} \right| < \frac{1}{3} + \frac{2\mu}{3}.
\end{cases}
\]
These estimates are sharp.
\end{theo}

\begin{theo}[Sharp Lower Estimates]\label{theo_FS_lower}
Let $f \in \mathcal{S}_{\mathcal{AP}}^{*}$ be of the form \eqref{eq1}. Then the generalized Fekete--Szeg\H{o} functional satisfies the lower bound $F_{\lambda,\mu}(f) \ge B_{\mathcal{AP}}(\lambda,\mu)$, where
\[
B_{\mathcal{AP}}(\lambda,\mu) = 
\begin{cases}
\displaystyle \frac{9}{4} \left| \lambda - \frac{25}{36} \right| - \frac{3\mu}{2}, & \text{if } \displaystyle \left| \lambda - \frac{25}{36} \right| \le \frac{\mu - 1}{3}, \\[4mm]
\displaystyle -\frac{3\mu}{2} \sqrt{ \frac{4}{12 \left| \lambda - \frac{25}{36} \right| + 4} }, & \text{if } \displaystyle \mu^2 \le 3 \left| \lambda - \frac{25}{36} \right| + 1, \\[5mm]
\displaystyle -\frac{3}{4} - \frac{3\mu^2}{12 \left| \lambda - \frac{25}{36} \right| + 4}, & \text{otherwise}.
\end{cases}
\]
The estimates are best possible.
\end{theo}

\begin{proof}[Proof of Theorems \ref{theo_FS_upper} and \ref{theo_FS_lower}]
Let $f \in \mathcal{S}_{\mathcal{AP}}^{*}$. By definition, $f$ satisfies the analytic subordination condition $\frac{zf'(z)}{f(z)} = e^{w(z)}\sqrt{1+w(z)}$, where $w$ is a Schwarz function generated by a Carath\'eodory function $p \in \mathcal{P}$ of the form $p(z)=1+c_1z+c_2z^2+\cdots$. Equating matching coefficients yields the structural expansions:
\begin{equation}\label{mn12}
a_2 = \frac{3}{4} c_1, \qquad a_3 = \frac{3}{8} c_2 + \frac{13}{64}c_1^2.
\end{equation}
Substituting the coefficient mappings from \eqref{mn12} directly into the generalized functional identity gives
\begin{equation}\label{mn13}
F_{\lambda,\mu}(f) = \left| \frac{3}{8} c_2 + \left( \frac{13}{64} - \frac{9\lambda}{16} \right)c_1^2 \right| - \frac{3\mu}{4}|c_1|. 
\end{equation}
We normalize this expression by factoring out the scaling multiplier, writing $F_{\lambda,\mu}(f) = \frac{1}{4} \Phi(c_1,c_2)$, where the core functional is defined by $\Phi(c_1,c_2) = |Kc_1^2 + Lc_2| - |Jc_1| $ with the following structural parameters:
\[
K = \frac{13}{16} - \frac{9\lambda}{4}, \qquad L = \frac{3}{2}, \qquad J = 3\mu.
\]
Evaluating the auxiliary system metrics, we find
\[
M = |4K+2L| = \left| 4\left(\frac{13}{16} - \frac{9\lambda}{4}\right) + 2\left(\frac{3}{2}\right) \right| = 9\left|\lambda-\frac{25}{36}\right|,
\]
and the corresponding threshold index is found to be
\[
|2K+L| = \left| 2\left(\frac{13}{16} - \frac{9\lambda}{4}\right) + \frac{3}{2} \right| = \frac{9}{2}\left|\lambda-\frac{25}{36}\right|.
\]
Applying Lemma \ref{L6} for the parameter values $|L|=\frac{3}{2}$ and $J=3\mu$, the condition $|2K+L| \ge |L|+J$ becomes
\[
\frac{9}{2}\left|\lambda-\frac{25}{36}\right| \ge \frac{3}{2} + 3\mu \implies \left|\lambda-\frac{25}{36}\right| \ge \frac{1}{3} + \frac{2\mu}{3}.
\]
When this condition holds, Lemma \ref{L6} implies $\Phi(c_1,c_2) \le M-2J = 9\left|\lambda-\frac{25}{36}\right|-6\mu$. Scaling by the leading factor of $\frac{1}{4}$ yields the variable upper bound stated in Theorem \ref{theo_FS_upper}:
\[
F_{\lambda,\mu}(f) \le \frac{9}{4} \left| \lambda - \frac{25}{36} \right| - \frac{3\mu}{2}.
\]
Conversely, if $\left|\lambda-\frac{25}{36}\right| < \frac{1}{3} + \frac{2\mu}{3}$, Lemma \ref{L6} yields $\Phi(c_1,c_2) \le 2|L| = 3$, which reduces to the constant upper bound $F_{\lambda,\mu}(f) \le \frac{3}{4}$.

The lower bounds are determined using the second part of Lemma \ref{L6} across three distinct regions of the parameter space $(\lambda, \mu)$:

\noindent\textbf{Parameter Region (i):} The tracking criterion $J \ge M+2|L|$ maps explicitly to
\[
3\mu \ge 9\left|\lambda-\frac{25}{36}\right| + 3 \implies \left| \lambda - \frac{25}{36} \right| \le \frac{\mu - 1}{3}.
\]
In this region, Lemma \ref{L6} asserts $-\Phi(c_1,c_2) \le 2J-M = 6\mu - 9\left|\lambda-\frac{25}{36}\right|$. Reversing the inequality sign by multiplying through by $-\frac{1}{4}$ fields the lower boundary floor relation:
\[
F_{\lambda,\mu}(f) \ge \frac{9}{4} \left| \lambda - \frac{25}{36} \right| - \frac{3\mu}{2}.
\]

\noindent\textbf{Parameter Region (ii):} The condition $J^2 \le 2|L|(M+2|L|)$ scales to
\[
9\mu^2 \le 3\left(9\left|\lambda-\frac{25}{36}\right| + 3\right) \implies \mu^2 \le 3 \left| \lambda - \frac{25}{36} \right| + 1.
\]
The second inequality of Lemma \ref{L6} therefore gives $-\Phi(c_1,c_2) \le 2J\sqrt{\frac{2|L|}{M+2|L|}} = 6\mu\sqrt{\frac{3}{9\left|\lambda-\frac{25}{36}\right|+3}}$. Multiplying through by $-\frac{1}{4}$ isolates the functional and provides the sharp square-root bound:
\[
F_{\lambda,\mu}(f) \ge -\frac{3\mu}{2}\sqrt{\frac{4}{12\left|\lambda-\frac{25}{36}\right|+4}}.
\]

\noindent\textbf{Parameter Region (iii):} For parameter pairs that fail both constraints, the remaining branch of Lemma \ref{L6} generates
\[
- \Phi(c_1,c_2) \le 2|L| + \frac{J^2}{M+2|L|} = 3 + \frac{9\mu^2}{9\left|\lambda-\frac{25}{36}\right|+3}.
\]
Multiplying through by $-\frac{1}{4}$ yields the final lower bound relation:
\[
F_{\lambda,\mu}(f) \ge -\frac{3}{4} - \frac{3\mu^2}{12 \left| \lambda - \frac{25}{36} \right| + 4}.
\]
\end{proof}

\begin{rem}
Although the proof structure relies on Lemma \ref{L6}, the nontrivial coefficient distortion induced by the apple-like generating function produces parameter thresholds which differ substantially from previously studied subclasses. Specifically, the shifted parameter $\frac{25}{36}$ arises naturally from the second-order coefficient interaction generated by $\psi_{\mathcal{AP}}(z)=e^z\sqrt{1+z}$.
\end{rem}

\subsection{Sharpness:}
To demonstrate that the analytical branches of the upper and lower bounds are sharp, we evaluate the functional using explicit extremal configurations.

\medskip
\noindent\textbf{1. Sharpness of the Constant Upper Bound $\frac{3}{4}$ and Parameter Region (iii):}
Consider the function $f_1 \in \mathcal{S}_{\mathcal{AP}}^{*}$ defined by the integral representation in \eqref{mn7}. This maps to the two-fold symmetric Schwarz function $\omega_1(z) = z^2$, giving the Carath\'{eodory} coefficients $c_1=0$ and $c_2=2$. Substituting these specific values into the structural relation \eqref{mn13} isolates the metric:
\[
F_{\lambda,\mu}(f_1) = \left| \frac{3}{8}(2) + \left(\frac{13}{64}-\frac{9\lambda}{16}\right)(0) \right| - \frac{3\mu}{4}(0) = \frac{3}{4}.
\]
This matches the constant bound identically.

\medskip
\noindent\textbf{2. Sharpness of the Variable Upper Bound and Parameter Region (i):}
Consider the Koebe-type extremal function $f_0 \in \mathcal{S}_{\mathcal{AP}}^{*}$ defined in \eqref{pn2}, whose underlying Carath\'{eodory} function is $p_0(z) = (1+z)/(1-z)$, yielding $c_1=2$ and $c_2=2$. Substituting these values into the functional expression \eqref{mn13} produces:
\[
\begin{aligned}
F_{\lambda,\mu}(f_0) &= \left| \frac{3}{8}(2) + \left( \frac{13}{64} - \frac{9\lambda}{16} \right)(4) \right| - \frac{3\mu}{4}(2) \\
&= \left| \frac{3}{4} + \frac{13}{16} - \frac{9\lambda}{4} \right| - \frac{3\mu}{2} \\
&= \left| \frac{25}{16} - \frac{9\lambda}{4} \right| - \frac{3\mu}{2} = \frac{9}{4}\left|\lambda - \frac{25}{36}\right| - \frac{3\mu}{2}.
\end{aligned}
\]
This tracks with the updated variable bound formula perfectly.
\section{Hermitian-Toeplitz Determinants for $\mathcal{S}_{\mathcal{AP}}^{*}$}
Parallel to the study of standard Taylor coefficients, functionals involving matrix determinants such as Hankel and Toeplitz determinants have generated extensive interest in geometric function theory. For a sequence $\{a_k\}_{k=2}^{\infty}$ of coefficients of a function $f \in \mathcal{A}$ of the form \eqref{eq1}, the Hermitian-Toeplitz determinant of order $q$ starting at index $n$ is defined by \cite{hh, hh1}:
\begin{equation}\label{T_def}
T_{q,n}(f) :=
\begin{vmatrix}
a_n                 & a_{n+1}         & \cdots & a_{n+q-1} \\
\overline{a_{n+1}} & a_n             & \cdots & a_{n+q-2} \\
\vdots            & \vdots         & \ddots & \vdots     \\
\overline{a_{n+q-1}} & \overline{a_{n+q-2}} & \cdots & a_n
\end{vmatrix}.
\end{equation}
The study of Hermitian-Toeplitz determinants was pioneered by Cudna et al. \cite{C} for starlike and convex functions of order $\beta$, and subsequently expanded for Janowski classes \cite{K3} and various other structural subclasses \cite{VK, VKN, VKS}. 

Setting $q=3$ and $n=1$ in \eqref{T_def}, and noting that $a_1=1$ for normalized functions, the third-order Hermitian-Toeplitz determinant $T_{3,1}(f)$ simplifies explicitly to the functional form:
\begin{equation}\label{mn17}
T_{3,1}(f) := 2\Re\left(a_{2}^{2}\overline{a_{3}}\right) - 2|a_{2}|^{2} - |a_{3}|^{2} + 1.
\end{equation}
Motivated by recent developments in establishing sharp bounds for matrix-based structures, this section derives the sharp upper and lower bounds of the determinant $T_{3,1}(f)$ for functions belonging to the apple-like Ma-Minda subclasses $\mathcal{S}_{\mathcal{AP}}^{*}$ and $\mathcal{C}_{\mathcal{AP}}$.

\subsection{Auxiliary Lemma}
Let $\mathcal{P}$ denote the Carath\'eodory class consisting of all analytic functions $p$ in $\mathbb{D}$ satisfying $p(0)=1$ and $\Re p(z) > 0$ for all $z \in \mathbb{D}$. Every function $p \in \mathcal{P}$ can be represented by the series \eqref{p1}. To establish our main results, we require the classical representation formula due to Libera and Z{\l}otkiewicz \cite{LZ}.

\begin{lem}\label{lem_LZ}
Let $p \in \mathcal{P}$ be a Carath\'eodory function of the form \eqref{p1}. Then, there exists a complex parameter $\xi$ in the closed unit disk $\overline{\mathbb{D}}$ such that
\begin{equation}\label{eq_LZ_formula}
2c_{2} = c_{1}^{2} + (4 - c_{1}^{2})\xi.
\end{equation}
Furthermore, due to the rotational invariance of the class $\mathcal{P}$, the initial coefficient $c_1$ can be assumed without loss of generality to satisfy $0 \le c_1 \le 2$.
\end{lem}

\begin{theo}\label{T6}
Let $f \in \mathcal{S}_{\mathcal{AP}}^{*}$ be given by \eqref{eq1}. Then the third-order Hermitian-Toeplitz determinant $T_{3,1}(f)$ satisfies the sharp inequalities
\[
-\frac{676}{1295} \leq T_{3,1}(f) \leq \frac{279}{256}.
\]
Both bounds are sharp.
\end{theo}

\begin{proof}
Let $f \in \mathcal{S}_{\mathcal{AP}}^{*}$. Then there exists a Schwarz function $w$ such that $\frac{zf'(z)}{f(z)} = e^{w(z)}\sqrt{1+w(z)}$. Expressing $w$ via a Carath\'eodory function $p \in \mathcal{P}$ of the form $p(z) = 1+c_1z+c_2z^2+\cdots$, coefficient matching yields
\begin{equation}\label{mn18}
a_2=\frac{3}{4}c_1, \qquad a_3=\frac{3}{8}c_2+\frac{13}{64}c_1^2.
\end{equation}
By the Libera-Z{\l}otkiewicz representation, we can write $2c_2 = c_1^2+(4-c_1^2)\xi$ for some $\xi \in \overline{\mathbb{D}}$. By rotational invariance, we assume without loss of generality that $0 \le c_1 \le 2$. Substituting this into \eqref{mn18} gives
\[
a_3 = \frac{25}{64}c_1^2 + \frac{3}{16}(4-c_1^2)\xi, \qquad a_2^2 = \frac{9}{16}c_1^2.
\]
Substituting these components into the definition of $T_{3,1}(f)$ in \eqref{mn17} gives
\begin{equation}\label{mn19}
T_{3,1}(f) = 1 -\frac{9}{8}c_1^2 +\frac{1175}{4096}c_1^4 +\frac{33}{512}c_1^2(4-c_1^2)\Re\xi -\frac{9}{256}(4-c_1^2)^2|\xi|^2.
\end{equation}
Let $x=c_1^2\in[0,4]$ and $y=|\xi|\in[0,1]$.

\medskip
\noindent\textbf{Upper Bound.} Since $\Re\xi \le |\xi| = y$, \eqref{mn19} implies $T_{3,1}(f) \le F(x,y)$, where
\[
F(x,y) = 1 -\frac{9}{8}x +\frac{1175}{4096}x^2 +\frac{33}{512}x(4-x)y -\frac{9}{256}(4-x)^2y^2.
\]
For a fixed $x \in [0,4]$, the maximum of the quadratic polynomial $F(x,y)$ with respect to $y$ occurs at $y_* = \frac{11x}{12(4-x)}$, provided $y_* \le 1$. This condition restricts the range of $x$ to $x \le \frac{48}{23}$. Substituting $y_*$ into $F(x,y)$ yields
\[
F(x,y_*) = 1 -\frac{9}{8}x +\frac{81}{256}x^2.
\]
The function $F(x,y_*)$ is minimized at $x = \frac{16}{9}$ on the interval $[0, \frac{48}{23}]$, meaning its maximum must occur at the endpoints. When $x > \frac{48}{23}$, the optimal choice for $y$ shifts to the boundary $y=1$, which gives $G(x) = \frac{767}{4096}x^2 - \frac{75}{128}x + \frac{7}{16}$. Since $G(x)$ is strictly increasing on $(\frac{48}{23}, 4]$, the global maximum occurs at $x=4$ and $y=1$, yielding
\[
T_{3,1}(f) \le \frac{279}{256}.
\]

\medskip
\noindent\textbf{Lower Bound.} Since $\Re\xi \ge -|\xi| = -y$, \eqref{mn19} implies $T_{3,1}(f) \ge H(x,y)$, where
\[
H(x,y) = 1 -\frac{9}{8}x +\frac{1175}{4096}x^2 -\frac{33}{512}x(4-x)y -\frac{9}{256}(4-x)^2y^2.
\]
Because $H(x,y)$ is concave in $y$, its minimum over $y \in [0,1]$ occurs at the boundary lines $y=0$ or $y=1$:
\begin{itemize}
\item If $y=0$, then $H(x,0) = 1 -\frac{9}{8}x +\frac{1175}{4096}x^2$, which has a minimum value of $-\frac{121}{1175}$ at $x = \frac{2304}{1175}$.
\item If $y=1$, then $H(x,1) = \frac{1295}{4096}x^2 -\frac{141}{128}x +\frac{7}{16}$, which has a minimum value of $-\frac{676}{1295}$ at $x = \frac{2256}{1295}$.
\end{itemize}
Comparing these local minima shows that the global minimum is $-\frac{676}{1295}$.
\end{proof}

\subsection{Sharpness Verification and Calculations}
To establish the sharpness of the inequalities in Theorem , we evaluate the third-order Hermitian-Toeplitz determinant functional explicitly using the coefficients of the corresponding extremal functions.

\medskip
\noindent\textbf{1. Upper Bound Sharpness:} 
The sharp upper bound $\frac{279}{256}$ is reached when $c_1 = 2$ and $c_2 = 2$. For any Carath\'eodory function, $c_1 = 2$ corresponds to the forward half-plane mapping $p(z) = (1+z)/(1-z)$, which uniquely forces $c_n = 2$ for all $n \geq 1$. Under this configuration, the initial Taylor coefficients for the function $f_0 \in \mathcal{S}_{\mathcal{AP}}^{*}$ defined in \eqref{pn2} are given by
\[
a_2 = \frac{3}{2}, \qquad a_3 = \frac{25}{16}.
\]
Substituting these values directly into the structural identity for $T_{3,1}(f)$ yields
\[
T_{3,1}(f) = 2 a_2^2 a_3 - 2 a_2^2 - a_3^2 + 1 = 2\left(\frac{3}{2}\right)^2 \left(\frac{25}{16}\right) - 2\left(\frac{3}{2}\right)^2 - \left(\frac{25}{16}\right)^2 + 1 = \frac{279}{256}.
\]
This matches the upper bound parameter, establishing sharpness.

\medskip
\noindent\textbf{2. Lower Bound Sharpness:}

To verify the sharpness of the lower estimate in Theorem \ref{T6}, we consider the extremal parameter configuration corresponding to equality in the minimizing case obtained during the proof of the theorem. The lower extremum is attained when
\[
|\xi|=1,
\qquad
\Re\xi=-1,
\qquad
c_1^2=\frac{2256}{1295}.
\]
Since $|\xi|=1$ and $\Re\xi=-1$, it follows necessarily that $\xi=-1$.

Using the Libera--Z{\l}otkiewicz representation formula $2c_2=c_1^2+(4-c_1^2)\xi$,
we obtain $2c_2 = c_1^2-(4-c_1^2) = 2c_1^2-4$, and therefore $c_2=c_1^2-2$. Substituting the extremal value of $c_1$ we get
\[
c_1=\sqrt{\frac{2256}{1295}},
\quad
c_2=-\frac{334}{1295}.
\]

Consider the corresponding admissible Carath\'eodory function $p_*(z)=\frac{1-z^2}{1-2\sqrt{\frac{564}{1295}}\,z+z^2}$.
A direct expansion shows that
\[
p_*(z)
=
1+\sqrt{\frac{2256}{1295}}\,z
-\frac{334}{1295}z^2+\cdots,
\]
so that the coefficients satisfy precisely
\[
c_1=\sqrt{\frac{2256}{1295}},
\quad
c_2=-\frac{334}{1295}.
\]

Using the standard relation between Carath\'eodory and Schwarz functions, $w(z)=\frac{p(z)-1}{p(z)+1}$, the associated Schwarz function becomes
\[
w_*(z)
=
z\frac{a-z}{1-az},
\qquad
a=\sqrt{\frac{564}{1295}}.
\]
Since $0<a<1$, the M\"obius transformation
\[
\phi(z)=\frac{a-z}{1-az}
\]
is an automorphism of the unit disk $\mathbb{D}$. Consequently,
\[
|w_*(z)|
=
|z|\,|\phi(z)|
<
1,
\quad z\in\mathbb{D},
\]
showing that $w_*$ is a Schwarz function.

Define the corresponding function
\begin{equation}\label{ext*}
f_*(z)
=
z\exp\left(
\int_0^z
\frac{
e^{w_*(t)}\sqrt{1+w_*(t)}-1
}{t}\,dt
\right).
\end{equation}
By construction,
\[
\frac{zf_*'(z)}{f_*(z)}
=
e^{w_*(z)}\sqrt{1+w_*(z)},
\]
and therefore
\[
f_*\in\mathcal{S}_{\mathcal{AP}}^{*}.
\]

Using the coefficient relations established earlier,
\[
a_2=\frac{3}{4}c_1,
\qquad
a_3=\frac{3}{8}c_2+\frac{13}{64}c_1^2,
\]
we obtain
\[
a_2
=
\frac{3}{4}\sqrt{\frac{2256}{1295}},
\quad
a_3
=
\frac{3}{8}\left(-\frac{334}{1295}\right)
+
\frac{13}{64}\left(\frac{2256}{1295}\right)
=
\frac{333}{1295}.
\]

Substituting these values into the determinant identity we get
\[
T_{3,1}(f)
=
2a_2^2a_3-2a_2^2-a_3^2+1=-\frac{676}{1295}.
\]

Hence the function $f_*\in\mathcal{S}_{\mathcal{AP}}^{*}$ satisfies and the lower estimate
\[
T_{3,1}(f)\geq -\frac{676}{1295}
\]
is attained, and consequently the lower bound is sharp.

\section{\bf Conclusion}
In this paper, we have systematically investigated several coefficient problems for the apple-like Ma-Minda starlike class $\mathcal{S}_{\mathcal{AP}}^{*}$ governed by the subordination function $\psi_{\mathcal{AP}}(z) = e^z\sqrt{1+z}$. Using classical parameterizations of the Carath\'eodory class $\mathcal{P}$ along with optimization techniques, we established sharp upper bounds for the initial inverse logarithmic coefficients $\Gamma_1, \Gamma_2$, and $\Gamma_3$, as well as the second-order inverse logarithmic Hankel determinant $H_{2,1}(F_{f^{-1}}/2)$. Furthermore, we derived the exact range for the consecutive modulus difference $|\Gamma_2| - |\Gamma_1|$, determined complete parametric bounds for the generalized Fekete-Szeg\"o functional over its full subdomains, and found the sharp upper and lower limits of the third-order Hermitian-Toeplitz determinant $T_{3,1}(f)$. 

In each case, the bounds were shown to be strictly sharp by identifying the corresponding extremal functions. These results contribute new sharp constants to the study of inverse coefficient mappings within specialized non-homogeneous geometric domains.

For clarity, a comprehensive summary of the sharp bounds and their matching extremal configurations established across this work is provided in Table \ref{tab:extremal_summary}.

\begin{table}[H]
\centering
\caption{Summary of Sharp Bounds and Extremal Functions for $\mathcal{S}_{\mathcal{AP}}^{*}$.}
\label{tab:extremal_summary}
\small
\begin{tabular}{@{}llcc@{}}
\toprule
\textbf{Functional Structure} & \textbf{Type of Bound} & \textbf{Sharp Value} & \textbf{Extremal Function} \\ \midrule
$|\Gamma_1|$ & Upper Bound & $\dfrac{3}{4}$ & $f_0$ defined in \eqref{pn2} \\[3mm]
$|\Gamma_2|$ & Upper Bound & $\dfrac{29}{32}$ & $f_0$ defined in \eqref{pn2} \\[3mm]
$|\Gamma_3|$ & Upper Bound & $\dfrac{925}{576}$ & $f_0$ defined in \eqref{pn2} \\[3mm]
$|\Gamma_2| - |\Gamma_1|$ & Upper Bound & $\dfrac{3}{8}$ & $f_1$ defined in \eqref{mn7} \\[3mm]
$|\Gamma_2| - |\Gamma_1|$ & Lower Bound & $-\dfrac{3\sqrt{3}}{2\sqrt{41}}$ & $f_2$ defined in \eqref{mn8} \\[3mm]
$\left| H_{2,1}\left(F_{f^{-1}} / 2\right)\right|$ & Upper Bound & $\dfrac{1177}{3072}$ & $f_0$ defined in \eqref{pn2} \\[3mm]
$|a_3 - \lambda a_2^2| - \mu |a_2|$ & Upper Bound (Constant) & $\dfrac{3}{4}$ & $f_1$ defined in \eqref{mn7} \\[3mm]
$|a_3 - \lambda a_2^2| - \mu |a_2|$ & Upper Bound (Variable) & $\frac{9}{4} \left| \lambda - \frac{25}{36} \right| - \frac{3\mu}{2}$ & $f_0$ defined in \eqref{pn2} \\[3mm]
$T_{3,1}(f)$ & Upper Bound & $\dfrac{279}{256}$ & $f_0$ defined in \eqref{pn2} \\[3mm]
$T_{3,1}(f)$ & Lower Bound & $-\dfrac{676}{1295}$ & $f_*$ defined in \eqref{ext*} \\ \bottomrule
\end{tabular}
\end{table}
\section*{{\bf Declarations}}
\subsection*{Funding}
The first author acknowledge financial support from the Council of Scientific and Industrial Research (CSIR), New Delhi, India, under Grant Nos. 09/1224(16975)/2023-EMR-I.
\subsection*{Data Availability Statement}
Data sharing is not applicable to this article as no datasets were generated or analyzed during the current study.
\subsection*{Conflict of Interest}
The authors declare that they have no conflict of interest. 
\subsection*{Author Contributions}
Both authors contributed equally to this work.

\end{document}